 \documentclass[final,12 pt, 3p]{elsarticle}

 \usepackage{graphics}
 \usepackage{verbatim }
 \usepackage{graphicx}
 \usepackage{epsfig}

\usepackage{amssymb}
\usepackage{amsmath}
\usepackage{amsthm}

   \newtheorem{theorem}{Theorem}[section]

   \newtheorem{corollary}[theorem]{Corollary}

   \newtheorem{definition}[theorem]{Definition}

\let\today\relax
\journal{Dynamical Systems}

\begin{document}

\begin{frontmatter}



\title{{Positive Invariance Condition for Continuous Dynamical Systems Based on Nagumo Theorem}
}
\author{Yunfei Song}

\address{Quant Strategies Group, BA Securities\\
\large\emph{December 1, 2020}}

\begin{abstract}
In this paper,  we obtain sufficient and necessary conditions of some classical convex sets
as positively invariant sets for a continuous dynamical system, namely positive invariance conditions. The approach is based on Nagumo Theorem by deriving the tangent cones of these sets. We also propose approaches using optimization theory and models to
verify the existence of these sufficient and necessary conditions.
\end{abstract}

\begin{keyword}

Dynamical System, Invariant Set, Polyhedral Set, Ellipsoid,
Lorenz Cone.
\end{keyword}

\end{frontmatter}

\section{Introduction}

Dynamical system has a wide range of applications in the real world. Positively invariant set is a key concept in dynamical system.  A positively invariant set of a dynamical system is described as when the system emits from the set, it will always stay in the set. Invariant set is intuitively considered as the attracting region, also namely safety area, of the dynamical system. The applications of positively invariant set refer to \cite{Blanchini2, Blanchini, boyd}. Given a set and a dynamical system, to verify if the set is an invariant set for the system is a key problem in this area. Such verification criteria are usually used to construct the maximal invariant set, i.e., the maximal safety, in a controlled dynamical system.  

Recently, some excellent surveys on the theoretical results and applications of invariant sets are published, e.g., \cite{Blanchini, songthesis}. For specific classical sets, the invariance condition, i.e., the sufficient and necessary condition such that a given set is an invariant set for a given dynamical system, are widely studied. For polyhedral sets,  one may refer to \cite{caste, dorea, song1, valch} for various invariance conditions for linear continuous and discrete dynamical system.  For quadratic type of sets, e.g., ellipsoidal and second order conic sets, one may refer to \cite{boyd, loewy, stern,Vander}. For general convex set and nonlinear system, one may refer to a novel unified approach to derive invariance
conditions for polyhedra, ellipsoids, and cones is presented in \cite{song1}.  The connection between discrete and continuous dynamical systems for preserving the invariance of a set is studied, e.g., \cite{song3, song2}. Construction of some invariant sets for a given system is also an interesting topic in this area, e.g., \cite{horva, songdikin}.

In this paper,  we derive the sufficient and necessary conditions of some classical convex sets
as positively invariant sets for a continuous dynamical system. These conditions are referred to as positive invariance condition for simplicity.  The candidates of the sets are polyhedra, ellipsoids and cones. The approach is primarily based on Nagumo Theorem \cite {Blanchini, nagu}, which yields the positive invariance conditions to be deriving the tangent cones of these sets. We also propose approaches using optimization theory and models to
verify the existence of these positive invariance conditions. The novelty of this paper is that we applied the theoretical result Nagumo Theorem into specific sets and dynamical system, as well as deriving the new positive invariance conditions or same invariance condition by using this new method. Also, the technique using optimization theory and algorithm is novel and the link between invariant set and optimization is built up. 

\emph{Notation and Conventions.} In this paper, we use the following
notation and conventions to avoid unnecessary repetitions, e.g.,  \cite {bellm, horn}
\begin{itemize}
  \item The inertia of a matrix is denoted by inertia$\{Q\}=\{a,b,c\}$ that indicates the number
of positive, zero, and negative eigenvalues of the matrix $Q$, respectively.
  \item The basis in $\mathbb{R}^n$ is denoted by $e^1=(1,0,...,0)^T, e^2=(0,1,...,0)^T,...,
  e^n=(0,0,...,1)^T$. And we let $e=(1,1,...,1)^T.$
  \item The nonnegative quadrant of $\mathbb{R}^n$ is denoted by
  $\mathbb{R}^n_+$, i.e., any coordinates of $x\in \mathbb{R}^n_+$
  is nonnegative.
  \item Let a vector $v\in \mathbb{R}^n$, we use
  $v({\uparrow}_{k}\alpha)$ to denote the $k$-th entry of $v$ is replaced by
  $\alpha$, i.e., $v({\uparrow}_{k}\alpha)=(v_1,...,v_{k-1},\alpha,v_{k+1},...,v_n)^T.$
\end{itemize}

The paper is structured as follows: Section 

\section{Fundamental  Definitions}
\subsection{Invariant Set and Nagumo Theorem}
We consider the continuous dynamical system, which is
also named \emph{initial value problem (IVP)}, as follows:
\begin{equation}\label{eqn:dy1}
\dot{x}(t)=f(t,x), ~~ t\geq 0
\end{equation}
where $x(t)\in \mathbb{R}^{n}$ are the state variables,  $t$ is the time variable,  and $f(t,x)$ is a real 
valued continuous function. For simplicity, we denote $(t,x)\in
\mathbb{R}_+\times\mathbb{R}^n$.\\

We now introduce the definition of positively invariant set of
a dynamical system. 

\begin{definition}
Let us assume $\mathcal{S}$ be a set in $\mathbb{R}^n$. The set
$\mathcal{S}$ is called a \textbf{positively invariant set} of the dynamical
system (\ref{eqn:dy1}) if $x(0)\in \mathcal{S}$ implies $x(t)\in
\mathcal{S}$ for all $t\geq0.$
\end{definition}

Positively invariant set is also named as forward invariant set. For simplicity, we use invariant set
to represent positively invariant set. In other words, an invariant set is a set that once the trajectory of the system enters the set, then it will never leave the set in the future. 
One example of an invariant set\footnote{The solution of the system is $x(t)=x(0)e^{At}$.  For any $x(0)$ in the span space, then $x(0)$ can be 
represented as $\sum_{i}^k \ell_i v_i$, where $\{v_i\}$ are the eigenvectors of $A$ and $\{\ell_i\}$ are some coefficients. Note that $e^{At}v_i=e^{\lambda_it}v_i$, where $\lambda_i$
is the eigenvalue corresponding to $v_i$, it is easy to derive the proof. } of the linear system $\dot{x}(t)=Ax(t)$, where $A$ is a real matrix, is the span space of all eigenvectors of the matrix $A$. 

A fundamental characterization of a close and convex set to be an invariant set for a
continuous system is proposed by Nagumo \cite{Blanchini, nagu}. 

\begin{theorem}\label{nagumo}
\emph{\textbf{Nagumo  \cite{Blanchini,nagu}:}} Let
$\mathcal{S}\subseteq\mathbb{R}^n$ be a closed convex set, and
assume that the system $\dot{x}(t)=f(t, x)$, where
$f:\mathbb{R}_+\times \mathbb{R}^n\rightarrow \mathbb{R}$ is a
continuous mapping, has a globally unique solution for every
initial point $x(t_0)\in \mathcal{S}$. Then $\mathcal{S}$ is an
invariant set for this system if and only if
\begin{equation}\label{cond3}
f(t,x)\in \mathcal{T}_\mathcal{S}(x),\text{  for all }  x\in
\partial \mathcal{S},
\end{equation}
where  $\mathcal{T}_\mathcal{S}(x)$ is the tangent cone of
$\mathcal{S}$ defined as follows:
\begin{equation}\label{tancon}
\mathcal{T}_\mathcal{S}(x)=\left\{y\in
\mathbb{R}^n\Big|\lim_{t\rightarrow0}\inf\frac{\|x+ty,\mathcal{S}\|}{t}=0\right\}.
\end{equation}
\end{theorem}

Note that the condition that the set is closed and convex is critical in this theorem. 
In fact,  this theorem can be simply illustrated in a
geometrical way: for any trajectory that emits from
$\mathcal{S}$, one only needs  to consider the property of this
trajectory hits the boundary $\partial\mathcal{S}$. We can see that condition (\ref{cond3}) ensures  the
trajectory points inside $\mathcal{S}$ since $f(t,x)$
is the derivative of the trajectory at $x$, thus $x(t)$ will stay in $\mathcal{S}$. Also,
there is no requirement that the set $\mathcal{S}$ needs a
specific form such that the theorem holds, thus this theorem is a general result. 
In this paper, we will apply Nagumo theorem on specific types of sets to derive the sufficient and necessary 
conditions such that the set is an invariant set for the continuous system (\ref{eqn:dy1}).

\subsection{Convex Sets}\label{conset}
In this subsection, we introduce the concepts of a family of convex
sets which are considered as invariant sets for dynamical systems.
In particular, these convex sets are  polyhedra, polyhedral cones,
ellipsoids, and Lorenz cones. These types of sets are common and used in many areas. 

A \emph{polyhedron} has two ways to define. The first way is given as
the intersection of a finite number of half-spaces as follows:
\begin{equation}\label{poly1}
\mathcal{P}=R[G,b ]=\{x\in \mathbb{R}^n~|~Gx\leq b\},
\end{equation}
where $G\in \mathbb{R}^{m\times n} $ and $ b\in \mathbb{R}^m$. The second way is given as the convex combination of a finite number of points
and a conic combination of some vectors as follows:
\begin{equation}\label{poly2}
\mathcal{P}=\Big\{x\in \mathbb{R}^n~|~x=\sum_{i=1}^{\ell_1}\theta_i
x^i+\sum_{j=1}^{\ell_2}\hat{\theta}_j\hat{x}^j,~
\sum_{i=1}^{\ell_1}\theta_i=1, \theta_i\geq0, \hat{\theta}_j\geq
0\Big\},
\end{equation}
where $x^1,...,x^{\ell_1},\hat{x}^1,...,\hat{x}^{\ell_2}\in
\mathbb{R}^n$.  A special type of polyhedra when it is bounded, i.e., $\ell_2=0$ in (\ref{poly2}), is referred  to as
\emph{polytope}.

A \emph{polyhedral cone} with origin as its vertex can be considered
as a special class of polyhedra, therefore we define polyhedral cone
as follows
\begin{equation}\label{polycone1}
\mathcal{C_P}=R[G,0]=\{x\in \mathbb{R}^n~|~Gx\leq 0\},
\end{equation}
or equivalently
\begin{equation}\label{polycone2}
\mathcal{C_P}=\Big\{x\in
\mathbb{R}^n~|~x=\sum_{j=1}^{\ell}\hat{\theta}_j\hat{x}^j,~
\hat{\theta}_j\geq 0\Big\},
\end{equation}
where $G\in \mathbb{R}^{m\times n}$, and
$\hat{x}^1,...,\hat{x}^{\ell}\in \mathbb{R}^n$. In particular, the
positive quadrant in $\mathbb{R}^n$, i.e., all coordinates are
nonnegative, which therefore denoted by $\mathbb{R}^n_+$, is a
special polyhedral cone and highly interesting as it has tremendous
scientific and engineering applications.

Since an arbitrary ellipsoid is equivalent to an ellipsoid with origin
as its center by a shifting transformation,  we consider only an
\emph{ellipsoid} centered at origin defined as follows:
\begin{equation}\label{elli}
\mathcal{E}=\{x\in\mathbb{ R}^n ~|~ x^TQx\leq 1\},
\end{equation}
where  $Q\in \mathbb{R}^{n\times n}$ is a symmetric positive
definite matrix.

A \emph{Lorenz cone}  is  also refereed to as ice cream cone, or
second order cone.  Similar to the case of ellipsoids, an arbitrary
Lorenz cone is equivalent to an Lorenz cone with vertex at origin by
a shifting transformation, therefore we only consider  a Lorenz cone
with vertex at origin defined as follows:
\begin{equation}\label{ellicone}
\mathcal{C_L}=\{x\in \mathbb{R}^n~|~x^TQx\leq 0,~ x^TQu_n\leq0\},
\end{equation}
where $Q\in \mathbb{R}^{n\times n}$ is a symmetric nonsingular
matrix with only one negative eigenvalue $\lambda_n$. Thus, we have
${\rm inertia}\{Q\}=\{n-1,0,1\}$, which yields that there exists an
orthonormal basis\footnote{Recall that orthonormal basis means that
$u_i^Tu_j=\delta_{ij}$, where $u_i$ is the eigenvector that
corresponds to $\lambda_i$ and $\delta_{ij}$ is Kronecker delta
function.} $ U=[u_1,u_2,...,u_n]$ such that
\begin{equation}
Q=U\Lambda^{\frac{1}{2}}\tilde{I}\Lambda^{\frac{1}{2}}U^T,
\end{equation}
where
$\Lambda^{\frac{1}{2}}=\text{diag}\{\sqrt{\lambda_1},...,\sqrt{\lambda_{n-1}},\sqrt{-\lambda_n}\}$
and $\tilde{I}=\text{diag}\{1,...,1,-1\}$. If we further operate an
appropriate orthogonal transformation to $\mathcal{C_L}$, then a
Lorenz cone with vertex at origin and axis at a coordinate axis,
which refers to as \emph{standard Lorenz cone} that denoted by
$\mathcal{C_L^*}$, is generated and equivalent  to $\mathcal{C_L}.$
In particular,  we have $ \mathcal{C_L^*}=\{x\in \mathbb{R}^n|
x^T\tilde{I}x\leq 0, x^T\tilde{I}e_n\leq0\}, $ where
$e_n=(0,...,0,1)^T.$

\section{Invariance Conditions}
\subsection{Tangent Cones}
In this subsection, we will derive the formula of the tangent
cones of polyhedra, polyhedral cone, ellipsoid, and Lorenz cone.
According to Nagumo Theorem \ref{nagumo}, the tangent cone is crucial for deriving the sufficient and 
necessary condition for an invariant set. For a given set $\mathcal{S}$, it is easy to
see that the tangent cone of a point $x$ in the interior of
$\mathcal{S}$ is $\mathbb{R}^n,$ thus we only consider the case when
$x$ is on the boudnary of $\mathcal{S}.$

\begin{theorem}\label{icthm1}
Let a polyhedron $\mathcal{P}$ (or a polyhedral cone
$\mathcal{C_P}$) be in the form of (\ref{poly1}) (or
(\ref{polycone1})). Assume $x$ is active at the $i_1$-th,
$i_2$-th,$...,i_k$-th constraints, i.e., $g_{i_1}^Tx=b_{i_1},...,
g_{i_k}^Tx=b_{i_k}$ (or $g_{i_1}^Tx=0,..., g_{i_k}^Tx=0$), then the
tangent cone at $x$ with respect to $\mathcal{P}$ (or
$\mathcal{C_P}$) is
\begin{equation}\label{iceq11}
\mathcal{T_P}(x) ~(\text{or } \mathcal{T_{C_P}}(x))~=\{y\in
\mathbb{R}^n~|~g_{i_j}^Ty\leq 0, j=1,2,...,k\}.
\end{equation}
\end{theorem}
\begin{proof}
For an arbitrary point $\hat{x}$ in $\mathcal{P}$, we have
$g_{i_j}^T\hat{x}\leq b_{i_j},$ which yields
$g_{i_j}^T(\hat{x}-x)\leq 0.$ Then choosing $y=\hat{x}-x$, we have
$x+ty\in \mathcal{P}$ for sufficient small $t$, which deduces that
$\|x+ty,\mathcal{P}\|=0.$ Then we complete the proof.
\end{proof}
Observing the formula of the tangent cone $\mathcal{T_P}$ (or
$\mathcal{T_{C_P}}$) in (\ref{iceq11}), we can find that
$\mathcal{T_P}$ (or $\mathcal{T_{C_P}}$) is a half space when $x$ is
active at a single constraint. As the nonnegative quadrant is a
special case of polyhedral cone, we have the following corollary.

\begin{corollary}
Let the nonnegative quadrant $\mathbb{R}_+^n$ be represented as
$\{x\in \mathbb{R}^n|x_{i}\geq0, i=1,2,...,n\}$.   Assume $x$ is on
the boundary of  $\mathbb{R}_+^n$, i.e., $x_{i_j}=0$ for some
$j=1,2,...,k$, then the tangent cone at $x$  with respect to
$\mathbb{R}_+^n$ is
\begin{equation}\label{3eq13}
\mathcal{T}_{\mathbb{R}^n_+}(x) ~=\{y\in \mathbb{R}^n|y_{i_j}\geq0,
j=1,2,...,k\}.
\end{equation}
\end{corollary}

We now turn to consider the second representation form of polyhedral sets given as in (\ref{poly2}) and (\ref{polycone2}). As an
arbitrary polyhedral set in the form of  (\ref{poly2}) is the union
of a polytope and a polyhedral cone,  we divide polyhedral sets into
two basic classes, i.e., polytope and polyhedral cone. The
difference between these two polyhedral sets is that a polytope is a
bounded set, while a polyhedral cone is unbounded. There is one common
characteristic between these two sets is that they are both
represented by a finite number of vectors, which refers to as
extreme points (or vertex) for polytope versus extreme ray for
polyhedral cone. Therefore, it suffices to consider the tangent cones
at these vectors instead of all points on the boundary. Similar to
the proof in Theorem \ref{icthm1}, the vector $x^j-x^i$ for any $j$
is in the tangent cone at $x^i$. Thus, the following  theorem is
immediate.

\begin{theorem}
Let a polytope $\mathcal{P}$ be in the form of (\ref{poly2}), i.e.,
all $\hat{\theta}_j=0$. Then the tangent cone at vertex $x^i$ with
respect to $\mathcal{P}$ is
\begin{equation}
\mathcal{T_P}(x^i)=\{y\in \mathbb{R}^n|y=\sum_{j=1,j\neq
i}^{\ell_1}\alpha_j^{(i)}(x^j-x^i),~~ \alpha_j^{(i)}\geq0, \text{
for } j=1,2,...,\ell_1\}.
\end{equation}
\end{theorem}

\begin{theorem}
Let a polyhedral cone $\mathcal{C_P}$ be in the form of
(\ref{polycone2}). Then the tangent cone at extreme ray $x^i$ with
respect to $\mathcal{C_P}$  is
\begin{equation}\label{3eq11}
\mathcal{T_{C_P}}(x^i)=\{y\in
\mathbb{R}^n|y=\hat{\alpha}_i^{(i)}x^i+\sum_{j=1,j\neq
i}^k\alpha_j^{(i)}x^j, ~~\hat{\alpha}_i^{(i)}\in \mathbb{R},
\alpha_j^{(i)}\geq0, \text{ for } j=1,2,...,\ell \}.
\end{equation}
\end{theorem}
\begin{proof}
One can prove that the tangent cone at extreme ray $x^i$ with
respect to $\mathcal{T_{C_P}}$ is $ \mathcal{T_P}(x^i)=\{y\in
\mathbb{R}^n|y={\alpha}_i^{(i)}x^i+\sum_{j=1,j\neq
i}^k\alpha_j^{(i)}(x^j-x^i), ~~{\alpha}_i^{(i)}\in \mathbb{R},
\alpha_j^{(i)}\geq0, \text{ for } j=1,2,..., \ell \}, $ which is
equivalent with (\ref{3eq11}) by letting
$\hat{\alpha}_i^{(i)}={\alpha}_i^{(i)}-\sum_{j\neq i}
\alpha_j^{(i)}.$
\end{proof}

As a matter of fact, the region covered by the tangent cone at an
extreme point with respect to a polytope contains no lines, while
the region covered by the tangent cone at an extreme ray with
respect to a polyhedral cone contains lines.

\begin{corollary}
Let the nonnegative quadrant ${\mathbb{R}_+^n}$ be represented as
$\{x\in \mathbb{R}^n|x=\sum_{i=1}^n\hat\theta_ie^i,
~\hat{\theta}_i\geq0\}$. Then the tangent cone at  $e_i$ with
respect to $\mathbb{R}_+^n$ is
\begin{equation}\label{3eq14}
\mathcal{T}_{\mathbb{R}_+^n}(x) ~=\{y\in \mathbb{R}^n|y_{j}\geq0,
j\neq i\}.
\end{equation}
\end{corollary}

We now analyze the tangent cones with respect to an ellipsoid and
a Lorenz cone. As these two sets are both represented by a
quadratic inequality, they can be considered simultaneously.
Moreover, the two sets are both smooth at the boundaries except the
vertex of Lorenz cone, thus  there exists a tangent space at the
boundaries except the vertex of Lorenz cone. Let us choose
ellipsoids as an example, the outer norm of an arbitrary point $x\in
\partial \mathcal{E} $ is $Qx$, then the tangent space at $x$ is
represented as $\{y\in \mathbb{R}^n|y^TQx=0\}$. Therefore, we have
the following theorem. Note that a similar result for Lorenz cone can refer to \cite{stern}. 

\begin{theorem} \cite{stern}
Let an ellipsoid  $\mathcal{E}$ (or a Lorenz cone $\mathcal{C_L}$)
be in the form of (\ref{elli}) (or (\ref{ellicone})). Assume $x$ is
on the boundary of $\mathcal{E}$ (or $\mathcal{C_L}$), then the
tangent cone at $x$ with respect to $\mathcal{E}$ (or
$\mathcal{C_L}$) is
\begin{equation}\label{3eq12}
\mathcal{T_E}(x) ~(\text{or } \mathcal{T_{C_L}})~=\{y\in
\mathbb{R}^n|y^TQx\leq0, \text{ for all }x\in \mathcal{E} ~( \text{or } \mathcal{C_L})\}.
\end{equation}
\end{theorem}

\subsection{Invariance Condition}

In this subsection, we will investigate the sufficient and
necessary conditions under which the involved convex sets in this
paper are invariant sets with respect to a dynamical system as shown
in (\ref{eqn:dy1}).

\begin{theorem}
Let a polyhedron $\mathcal{P}$ (or a polyhedral cone
$\mathcal{C_P}$) be in the form of (\ref{poly1}) (or
(\ref{polycone1})). Then $\mathcal{P}$ (or  $\mathcal{C_P}$) is an
invariant set with respect to the dynamical system (\ref{eqn:dy1})
if and only if any point $x$ on the boundary of $\mathcal{P}$ (or
$\mathcal{C_P}$)  holds  the following condition
\begin{equation}\label{iceq21}
g_{i_j}^Tf(t_0,x)\leq 0, j=1,2,...,k,
\end{equation}
where $x$ is active at at the $i_1$-th, $i_2$-th,$...,i_k$-th
constraints.
\end{theorem}

\begin{corollary}
Let the nonnegative quadrant $\mathbb{R}_+^n$ be represented as
$\{x\in \mathbb{R}^n|x_{i}\geq0, i=1,2,...,n\}$. Then
$\mathbb{R}_+^n$ is an invariant set with respect to the dynamical
system (\ref{eqn:dy1}) if and only if any point $x$ on the boundary
of $\mathbb{R}_+^n$  holds  the following condition
\begin{equation}\label{iceq22}
f_{i_j}(t_0,x)\geq0, j=1,2,...,k,
\end{equation}
where $x_{i_j}=0$ for $j=1,2,...,k$.
\end{corollary}

\begin{theorem}
Let a polytope $\mathcal{P}$ be in the form of (\ref{poly2}), i.e.,
all $\hat{\theta}_j=0$. Then $\mathcal{P}$ is an invariant set with
respect to the dynamical system (\ref{eqn:dy1}) if and only if for
any extreme point $x^i$, there exists nonnegative scalars
$\alpha_j^{(i)}\geq0, \text{ for } j\neq i, j=1,2,...,\ell_1$  such
that
 the following condition holds
\begin{equation}\label{iceq23}
f(t_0,x^i)=\sum_{j=1,j\neq i}^{\ell_1}\alpha_j^{(i)}(x^j-x^i).
\end{equation}
\end{theorem}

We now investigate the way to verify the existence of the
coefficients $\alpha_j^{(i)}$ in (\ref{iceq23}), which might be not
unique when $\ell_1>n$, i.e., the number of vertices is greater than
the dimension of points.  Condition (\ref{iceq23}) is equivalently
reformulated as
\begin{equation}\label{iceq231}
f(t_0,x^i)=\sum_{j\neq i}\alpha_j^{(i)}x^j-\left(\sum_{j\neq
i}\alpha_j^{(i)}\right)x^i=\sum_{j=1 }^{\ell_1}\alpha_j^{(i)}x^j,
\text{ with } \sum_{j=1 }^{\ell_1}\alpha_j^{(i)}=0,
\end{equation}
where $\alpha_i^{(i)}=-\sum_{j\neq i }\alpha_j^{(i)}$.

For the sake of simplicity, we denote $X=[x^1,x^2,...,x^{\ell_1}],$
 $\tilde{X}=[\tilde{x}^1,\tilde{x}^2,...,\tilde{x}^{\ell_1}]=[X^T,e]^T,$ $\tilde{f}=(f^T,0)^T,$ and
$\alpha^{(i)}=(\alpha_1^{(i)},\alpha_2^{(i)},...\alpha_{\ell_1}^{(i)})^T.$
Then two optimization models can be built to solve the coefficients
in (\ref{iceq23}).

The first model that essentially is a linear feasibility problem is
represented as follows:
\begin{equation}\label{iceq27}
\begin{aligned}
\min ~        & ~~0\\
\text{ s.t. } & ~\tilde{X}\alpha^{(i)}=\tilde{f}, ~~
\alpha_j^{(i)}\geq 0, ~j\neq i.
\end{aligned}
\end{equation}
As the objective function in (\ref{iceq27}) is a fixed number, the
optimization model (\ref{iceq27}) has optimal solution, which is not
necessary unique, if and only if the condition (\ref{iceq23}) has
solution. Without loss of generality, we assume $X$ has the property
that its column vectors are independent, and thereby $\tilde{X}$ as
well. Then the model (\ref{iceq27}) can be discussed into to two
cases: the number of rows of $\tilde{X}$ is greater than or equal to
the number of columns of $\tilde{X}$, and the number of rows of
$\tilde{X}$ is less than the number of columns of $\tilde{X}$. As a
matter of fact, the first case is trivial, since the equation
$\tilde{X}\alpha^{(i)}=\tilde{f}$, which assume it has solutions,
has an unique solution, which is explicitly represented as
$$
\alpha^{(i)}=(\tilde{X}^T\tilde{X})^{-1}\tilde{X}^T\tilde{f}.
$$
Then one only needs to check whether the solution satisfies the
condition that $\alpha_j^{(i)}\geq 0, j\neq i$. In the second case,
since the equation $\tilde{X}\alpha^{(i)}=\tilde{f}$ always has
solutions that might be unique or not, which means the optimization
model (\ref{iceq27}) is always feasible, it is hard to obtain the
solution of this equation directly. 

There are normally two ways to
solve the optimization model (\ref{iceq27}): pivot methods or
interior point methods (IPMs). The pivot methods, e.g., Simplex
algorithm\cite{bert}, Criss-Cross algorithm\cite{Terlaky}, have
exponential complexity. The basic idea of pivot methods is
enhancing, i.e, decreasing the objective function if it is a
minimization problem, the optimization problem along the edges of
the polyhedral set which is the feasible region of this optimization
problem by pivoting from one vertex to another. The IPMs
\cite{Terlaky1} have polynomial complexity. The basic idea of IPMs
is enhancing the optimization problem by along a central path
interior the feasible region. Although IPMs have polynomial
complexity, pivot methods shows great efficiency when the problem is
not very large for linear optimization. The advantage of IPMs
becomes significant when the size of the problem is large.

To solve the optimization model (\ref{iceq27}), one can also
consider its dual problem, which is also a linear method and
represented as
\begin{equation}
\begin{aligned}
\max~ & ~\tilde{f}^Ty\\
\text{ s.t. } &~(\tilde{x}^i)^Ty=0,\\
&~(\tilde{x}^j)^Ty\leq 0, ~~j\neq i,
\end{aligned}
\end{equation}
Then according to the knowledge in optimization, e.g., \cite{Roos},
the optimality condition that is the sufficient and necessary
condition of the existence of the optimal solution for an
optimization problem, by introducing the artificial variables
$s_j^{(i)}$, is
\begin{equation}
\tilde{X}\alpha^{(i)}=\tilde{f}, ~(\tilde{x}^i)^Ty=0,~~
(\tilde{x}^j)^Ty+s_j^{(i)}=0, ~~\alpha_j^{(i)},s_j^{(i)}\geq
0,~~j\neq i.
\end{equation}

The second model that is a quadratic optimization model is as
follows:
\begin{equation}\label{iceq232}
\begin{aligned}
\min~ & ~\small{\frac{1}{2}}\left\|X\alpha^{(i)}-f\right\|_2^2\\
\text{ s.t. } &~ e^T\alpha^{(i)}=0, ~~\alpha_j^{(i)}\geq0, ~j\neq i.\\
\end{aligned}
\end{equation}
One key difference between model (\ref{iceq27}) and (\ref{iceq232})
is that the former one might be infeasible, while the latter one is
always feasible, which implies the latter model always has optimal
solutions.  The objective function in (\ref{iceq232}) is the half of
the square of the distance between $f$ and $X\alpha^{(i)},$
therefore the optimal objective function value is exactly equal to 0
if model (\ref{iceq27})  is feasible. Otherwise, the model
(\ref{iceq232}) will output the point that is closest to the
polyhedral set defined by the vectors $x^1,...,x^{\ell_1}$, in which
case, the optimal objective function value is strictly positive. To
solve the quadratic model (\ref{iceq232}), we consider its
Lagrangian, which, by introducing the dual variables
$\eta^{(i)}=(\eta_1^{(i)},\eta_2^{(i)},...,\eta_{\ell_1}^{(i)})^T$
with $\eta_j^{(i)}\geq0$ for $j\neq i,$ is described as
\begin{equation}
L=\frac{1}{2}\left\|X\alpha^{(i)}-f\right\|_2^2+\eta_i^{(i)}e^T\alpha^{(i)}-\sum_{j\neq
i}\eta_j^{(i)}\alpha_j^{(i)}.
\end{equation}
The Karush-Kuhn-Tucker (KKT) condition \cite{nocedal} is normally
chosen as the first-order necessary optimality condition of  an
optimization problem. To verify the KKT condition of an optimization
problem, one needs to check whether this problem satisfies a so
called linear independence constraint qualification (LICQ) holds.
The LICQ is said to hold at a point $x$ if the gradients of the
active constraint are linearly independent at $x$. It is easy to
verify that optimization problem (\ref{iceq232}) holds LICQ. Then
the KKT condition of (\ref{iceq232}) is shown as follows:
\begin{eqnarray}
X^T(X\alpha^{(i)}-f)+\eta_i^{(i)}e-\eta^{(i)}(\uparrow_i0)&=&0\\
 e^T\alpha^{(i)}&=&0\\
\alpha_j^{(i)}&\geq&0, j\neq i\\
\eta_j^{(i)}\alpha_j^{(i)}&=&0, j\neq i
\end{eqnarray}
where $\eta^{(i)}(\uparrow_i0)$ denotes the $i$-th entry in
$\eta^{(i)}$ is replaced by $0.$

\begin{theorem}
Let a polyhedral cone $\mathcal{C_P}$ be in the form of
(\ref{polycone2}). Then $\mathcal{C_P}$ is an invariant set with
respect to the dynamical system (\ref{eqn:dy1}) if and only if for
any extreme ray $x^i$, there exists nonnegative scalars
$\alpha_j^{(i)}\geq0, \text{ for }j\neq i, j=1,2,...,\ell$, and
$\hat{\alpha}_i^{(i)}\in \mathbb{R},$ such that the following
condition holds
\begin{equation}\label{iceq24}
f(t_0,x^i)=\hat{\alpha}_i^{(i)}x^i+\sum_{j=1,j\neq
i}^k\alpha_j^{(i)}x^j.
\end{equation}
\end{theorem}

\begin{corollary}
Let the nonnegative quadrant ${\mathbb{R}_+^n}$ be represented as
$\{x\in \mathbb{R}^n|x=\sum_{i=1}^n\hat\theta_ie^i,
~\hat{\theta}_i\geq0\}$. Then ${\mathbb{R}_+^n}$ is an invariant set
with respect to the dynamical system (\ref{eqn:dy1}) if and only if
for any extreme ray $e^i$,  the following condition holds
\begin{equation}\label{iceq25}
f_j(t_0,e^i)\geq0, j\neq i.
\end{equation}
\end{corollary}

\begin{theorem}\label{icthm31}
Let an ellipsoid  $\mathcal{E}$ (or a Lorenz cone $\mathcal{C_L}$)
be in the form of (\ref{elli}) (or (\ref{ellicone})). Then
$\mathcal{E}$ (or $\mathcal{C_L}$) is an invariant set with respect
to the dynamical system (\ref{eqn:dy1}) if and only if any point $x$
on the boundary of $\mathcal{E}$ (or $\mathcal{C_L}$)   holds  the
following condition
\begin{equation}\label{iceq26}
x^TQf(t_0,x)\leq0.
\end{equation}
\end{theorem}

According to Theorem \ref{icthm31}, one has to check whether all
points on the boundary of an ellipsoid or a Lorenz cone satisfy
condition (\ref{iceq26}). But it is complicated if we directly
examine condition (\ref{iceq26}) along the boundary of an ellipsoid
or a Lorenz cone. We present an optimization method to solve this
problem. For an ellipsoid $\mathcal{E},$ we consider the following
optimization model.
\begin{equation}\label{iceq41}
\begin{aligned}
\max~ & ~x^TQf(t_0,x)\\
\text{ s.t. } &~ x^TQx=1,
\end{aligned}
\end{equation}
where $Q$ is a symmetric positive definite matrix. This is not a
convex problem, as the constraint is nonconvex.
\begin{equation}\label{iceq42}
\begin{aligned}
\max~ & ~x^TQf(t_0,x)\\
\text{ s.t. } &~ x^TQx=1,
\end{aligned}
\end{equation}

We consider the linear dynamical system, i.e., $f(t_0,x)=Ax$. Then
the optimization problem can be formulated as
\begin{equation}\label{iceq43}
\begin{aligned}
\min~ & ~-\frac{1}{2} x^T(A^TQ+QA)x\\
\text{ s.t. } &~ x^TQx=1,
\end{aligned}
\end{equation}
The Lagrangian of optimization problem (\ref{iceq43}) is as follows:
\begin{equation}
L = -\frac{1}{2}x^T(QA+A^TQ)x+\frac{\eta}{2} (x^TQx-1).
\end{equation}
It is easy to check the optimization problem (\ref{iceq43})
satisfies LICQ condition, thus the first order optimality contrition
(KKT condition) is
\begin{equation}\label{iceq44}
\begin{aligned}
(A^TQ+QA-\eta Q)x &=&0\\
 x^TQx&=&1
\end{aligned}
\end{equation}
It is easy to check $x$ is not equal to 0, and $A^TQ+QA-\eta Q$ has
to be singular. Also note that $(A^TQ+QA)x=\eta Qx$, which is
substituted into the objective function in (\ref{iceq43}), we have
$-\frac{1}{2} x^T(A^TQ+QA)x=-\frac{1}{2} \eta x^TQx=-\frac{1}{2}
\eta$, where we apply $x^TQx=1$. Therefore $\eta\leq0$ such that the
optimal objective function value is always nonnegative.

Now we consider the second order optimality condition, which can be
written as
\begin{eqnarray}
d^T(A^TQ+QA-\eta Q)d &\leq&0\label{iceq451}\\
 d^TQx&=&0\label{iceq452}\\
(A^TQ+QA-\eta Q)x &=&0\label{iceq453}\\
 x^TQx&=&1\label{iceq454}
\end{eqnarray}
We now show that (\ref{iceq451})-(\ref{iceq454}) yield $A^TQ+QA-\eta
Q\preceq0.$ Since $A^TQ+QA-\eta Q$ is singular, the  condition in
(\ref{iceq454}) can be replaced by $x\neq0.$ The we consider the
following two cases: if $d=Qx$ also satisfies condition
(\ref{iceq451}), then we can say that condition (\ref{iceq451})
satisfies for any $d\in \mathbb{R}^n$, which is equivalent to
$A^TQ+QA-\eta Q\preceq0.$ Otherwise, if $d=Qx$ dose not satisfy
condition (\ref{iceq451}), we have
\begin{equation}\label{iceq455}
(Qx)^T(A^TQ+QA-\eta Q)(Qx)>0.
\end{equation}
Now assume $A^TQ+QA-\eta Q\npreceq0,$ then there exists an nonzero
vector $\tilde{x}\in \mathbb{R}^n$, such that $(A^TQ+QA-\eta
Q)\tilde{x}=\lambda \tilde{x},$ where $\lambda >0.$ By multiplying
appropriate scalar for $\tilde{x}$, we can have the following
orthogonal decomposition of $\tilde{x}$,
\begin{equation}
\tilde{x}=\tilde{d}+Qx, \text{ where } \tilde{d}^TQx=0.
\end{equation}
Then $\tilde{d}$ satisfies condition (\ref{iceq451}). Substituting
$\tilde{d}=\tilde{x}-Qx$ into the left formula of condition
(\ref{iceq451}), we have
\begin{equation}\label{iceq456}
\lambda \|\tilde{x}\|^2-2\lambda (Qx)^T\tilde{x}+(Qx)^T(A^TQ+QA-\eta
Q)(Qx).
\end{equation}
Note that $\|\tilde{x}\|^2=\|\tilde{d}\|^2+\|Qx\|^2+2
(Qx)^T\tilde{x}$, we have $\|\tilde{x}\|^2>2 (Qx)^T\tilde{x}.$ Also,
applying (\ref{iceq455}) to (\ref{iceq456}), we have that the
formula in (\ref{iceq456}) is
 positive. This is a contradiction. Therefore, in this case, we also have $A^TQ+QA-\eta
 Q\preceq0$.

Since $A^TQ+QA-\eta Q$ is singular, the last condition in
(\ref{iceq456}) can be replaced by $x\neq0.$ By left multiplying
$d^T$ to the third condition in  (\ref{iceq456}), we have
$d^T(A^TQ+QA)d=0$. in one can prove that (\ref{iceq456}) is
equivalent with that
\begin{equation}\label{iceq46}
\begin{aligned}
d^T(A^TQ+QA-\eta Q)d &\leq&0\\
 x^T(A^TQ+QA-\eta Q)d&=&0\\
(A^TQ+QA-\eta Q)x &=&0\\
\end{aligned}
\end{equation}
where $x\neq0.$

For a Lorenz cone $\mathcal{C_L},$ we consider the following
optimization model.
\begin{equation}\label{iceq51}
\begin{aligned}
\max~ &~ x^TQf(t_0,x)\\
\text{s.t. } &~ x^TQx=1,\\
& ~x^TQu_n\leq 0,
\end{aligned}
\end{equation}

We consider the linear dynamical system, i.e., $f(t_0,x)=Ax$. Then
the optimization problem can be formulated as
\begin{equation}\label{iceq52}
\begin{aligned}
\min~ & -\frac{1}{2} x^T(A^TQ+QA)x\\
\text{ s.t. } &x^TQx=1\\
& x^TQu_n\leq 0
\end{aligned}
\end{equation}
The Lagrangian of (\ref{iceq52}) is
\begin{equation}
L = -\frac{1}{2}x^T(QA+A^TQ)x+\frac{\eta}{2} (x^TQx-1) + \alpha
x^TQu_n.
\end{equation}
The KKT condition is
\begin{equation}\label{iceq53}
\begin{aligned}
(A^TQ+QA-\eta Q)x -\alpha x^TQu_n&=&0\\
 x^TQx&=&1\\
 x^TQu_n&\leq&0\\
 \alpha&\geq&0\\
\alpha x^TQu_n&=&0
\end{aligned}
\end{equation}
We can also prove that $\eta\leq0$ as the discussion of ellipsoid.

Now we consider the second order optimality condition, which can be
written as
\begin{equation}\label{iceq55555}
\begin{aligned}
d^T(A^TQ+QA-\eta Q)d &\leq&0\\
 d^TQx&=&0\\
(A^TQ+QA-\eta Q)x -\alpha Qu_n&=&0\\
 x^TQx&=&1\\
 x^TQu_n&\leq&0\\
 \alpha&\geq&0\\
\alpha x^TQu_n&=&0
\end{aligned}
\end{equation}
If $x^TQu_n<0$, which implies $\alpha=0,$ then this  yields a
similar condition as ellipsoid. Thus, $A^TQ+QA-\eta Q\preceq0.$ If
$x^TQu_n=0, $ a similar argument will be applied to derive the conclusion. 

\section{Conclusion}

Positively invariant set is an important concept in dynamical system and has widely used in application in control.  In this paper, we investigate Nagumo Theorem and apply it into specific convex sets, e.g., polyhedra, ellipsoids and cones. Then we derive the sufficient and necessary conditions of some classical convex sets
as positively invariant sets for a continuous dynamical system. The method is to derive the tangent cones of these sets. We derive some new positive invariance conditions or similar invariance conditions obtained by other researchers. To verify the invariance conditions, we propose methods using optimization theory and models. The introduction of using optimization techniques brings a novel insight on studying invariant set for a continuous dynamical system.

\bibliographystyle{plain}
\bibliography{myref}

\end{document}